\documentclass[oneside,leqno]{amsart}

\usepackage{amsmath,amsfonts,amssymb}
\usepackage[english]{babel}
\usepackage[all]{xy}

\newtheorem{thm}{Theorem}[section]
\newtheorem{prop}[thm]{Proposition}

\theoremstyle{definition}
\newtheorem{defi}[thm]{Definition}
\newtheorem{rem}[thm]{Remark}

\newtheorem{question}[thm]{Question}

\newcommand{\N}{\mathbb{N}}
\newcommand{\Z}{\mathbb{Z}}

\newcommand{\R}{\mathbb{R}}
\newcommand{\C}{\mathbb{C}}

\newcommand{\T}{\mathbb{T}}

\newcommand{\st}{\;:\;}

\newcommand{\sspace}{\cdot}
\newcommand{\ssspace}{\cdot\cdot}

\DeclareMathOperator{\im}{i}

\DeclareMathOperator{\de}{d}
\DeclareMathOperator{\id}{id}

\DeclareMathOperator{\esp}{e}
\DeclareMathOperator{\GL}{GL}
\DeclareMathOperator{\SL}{SL}

\newcommand{\Cpf}{$\mathcal{C}^\infty$-pure-and-full}

\newcommand{\Cf}{$\mathcal{C}^\infty$-full}
\newcommand{\Cp}{$\mathcal{C}^\infty$-pure}
\newcommand{\pf}{pure-and-full}

\newcommand{\f}{full}
\newcommand{\p}{pure}
\newcommand{\del}{\partial}
\newcommand{\delbar}{\overline{\del}}

\title[On Cohomological Decomposability of Almost--K\"ahler Structures]
{On Cohomological Decomposability of Almost--K\"ahler Structures}

\author{Daniele Angella}
\address[Daniele Angella]{Dipartimento di Matematica\\
Universit\`{a} di Pisa \\
Largo Bruno Pontecorvo 5, 56127\\
Pisa, Italy}
\email{angella@mail.dm.unipi.it}

\author{Adriano Tomassini}
\address[Adriano Tomassini]{Dipartimento di Matematica e Informatica\\
Universit\`{a} di Parma \\
Parco Area delle Scienze 53/A, 43124 \\
Parma, Italy}
\email{adriano.tomassini@unipr.it}

\author{Weiyi Zhang}
\address[Weiyi Zhang]{University of Michigan \\
Department of Mathematics \\
1825 East Hall \\
Ann Arbor, MI 48109
}
\email{wyzhang@umich.edu}

\keywords{$\mathcal{C}^\infty$-pure-and-full structure; $J$-anti-invariant cohomology; almost-K\"ahler structure}
\thanks{The first and the second authors are partially supported by GNSAGA of INdAM.}
\subjclass[2000]{53C55; 53C25; 32G05}

\begin{document}

\begin{abstract}
We study the $J$-invariant and $J$-anti-invariant cohomological subgroups of the de Rham cohomology of a compact manifold $M$ endowed with an almost-K\"ahler structure $\left(J,\,\omega,\,g\right)$. In particular, almost-K\"ahler manifolds satisfying a Lefschetz type property, and solvmanifolds endowed with left-invariant almost-complex structures are investigated.
\end{abstract}

\maketitle

\begin{footnotesize}\begin{center}
{\sffamily
\parbox{8cm}{
First published in Proceedings of the American Mathematical Society in 2012, published by the American Mathematical Society.
}}
\end{center}\end{footnotesize}

\section*{Introduction}
Cohomological properties of compact complex, and, more in general, almost-complex, manifolds have been recently studied by many authors, see, e.g., \cite{angella-tomassini-3}, respectively \cite{draghici-li-zhang, draghici-li-zhang-1}, and the references therein. The study of the cohomology of almost-complex manifolds is motivated, in particular, by a question of Donaldson's, \cite[Question 2]{donaldson}, relating the tamed and compatible symplectic cones of a compact $4$-dimensional almost-complex manifold, see, e.g., \cite{li-zhang}, and by the analogous question arising for compact higher dimensional complex manifolds, see \cite[page 678]{li-zhang} and \cite[Question 1.7]{streets-tian}. (We recall that a symplectic structure $\omega$ on a manifold $M$ is said to \emph{tame} an almost-complex structure $J$ if $\omega_x\left(u_x,\,  J_xu_x\right)>0$ for any $x\in M$ and for any $u\in T_xM\setminus\{0\}$, and it is said \emph{compatible} with $J$ if $g:=\omega(\sspace,\, J\ssspace)$ is a $J$-Hermitian metric;
 in the latter case, the triple $\left(J,\, \omega,\, g\right)$ is called an \emph{almost-K\"ahler structure} on $M$.)

\smallskip

Following T.-J. Li and the third author, \cite{li-zhang}, an almost-complex structure $J$ on a $2n$-dimensional manifold $M$ is called {\em \Cpf}\ if
$$
H^2_{dR}(M;\R)=H_J^{(1,1)}(M)_\R\oplus H_J^{(2,0),(0,2)}(M)_\R\,,
$$
where $H_J^{(1,1)}(M)_\R$ and $H_J^{(2,0),(0,2)}(M)_\R$ denote the subgroups of $H^2_{dR}(M;\R)$ whose elements can be represented by forms of type $(1,1)$ and $(2,0)+(0,2)$ respectively. In the notation of T. Dr\v{a}ghici, T.-J. Li, and the third author, \cite{draghici-li-zhang}, $H_J^{(1,1)}(M)_\R=:H^+_J(M)$ and $H_J^{(2,0),(0,2)}(M)_\R=:H^-_J(M)$ are the {\em $J$-invariant} and the {\em $J$-anti-invariant cohomology subgroups} respectively.

In \cite[Theorem 2.3]{draghici-li-zhang}, T. Dr\v{a}ghici, T.-J. Li, and the third author proved that every almost-complex structure on a compact $4$-dimensional manifold is \Cpf. This is no more true in dimension greater than four, see, e.g., \cite[Example 3.3]{fino-tomassini}, see also \cite{angella-tomassini, angella-tomassini-1}.

The groups $H_J^{(1,1)}(M)_\R$ and $H_J^{(2,0),(0,2)}(M)_\R$ appear as a natural generalization of the Dolbeault cohomology groups to the non-integrable case, see, e.g., \cite[Proposition 2.1]{li-zhang}. In fact, compact K\"ahler manifolds are \Cpf, and, in this case, $H_J^{(1,1)}(M)_\R\simeq H^{1,1}_{\overline{\partial}}(M)\cap H_{dR}^2(M;\R)$ and $H_J^{(2,0),(0,2)}(M)_\R\simeq\left(H^{2,0}_{\overline{\partial}}(M)\oplus H^{0,2}_{\overline{\partial}}(M)\right)\cap H_{dR}^2(M;\R)$.

We remark that, on a compact complex manifold, other cohomologies can be defined, namely, the Bott-Chern and Aeppli cohomologies. In \cite{angella-tomassini-3}, the problem of cohomology decomposition in terms of the Bott-Chern cohomology groups is investigated, providing in particular a characterization of compact complex manifolds satisfying the $\del\delbar$-Lemma.

\smallskip

Compact K\"ahler manifolds being \Cpf, in this paper we are interested in the study of the cohomological subgroups $H_J^{(1,1)}(M)_\R$ and $H_J^{(2,0),(0,2)}(M)_\R$ for almost-K\"ahler manifolds.

On the one hand, A. Fino and the second author, \cite[Proposition 3.2]{fino-tomassini}, as well as T. Dr\v{a}ghici, T.-J. Li, and the third author, \cite[Proposition 2.8]{draghici-li-zhang}, proved that the almost-complex structure of a compact almost-K\"ahler manifold is \Cp. On the other hand, we prove the following result, showing therefore a difference between the integrable and the non-integrable cases.

\smallskip
\noindent {\bfseries Proposition \ref{prop:iwasawa-non-full}.\ }{\itshape
  Let $X:=\left. \Z\left[\im\right]^3 \right\backslash \left(\C^3,\,*\right)$ be the real manifold underlying the Iwasawa manifold. Then there exists an almost-K\"ahler structure $\left(J,\,\omega,\,g\right)$ on $X$ which is \Cp\ and non-\Cf.
  Furthermore, the {\em Lefschetz type operator} $\mathcal{L}_\omega:=\omega\wedge \sspace \colon\wedge^2M\to\wedge^{4}M$ of the almost-K\"ahler structure $\left(J,\,\omega,\, g\right)$ does not take $g$-harmonic $2$-forms to $g$-harmonic $4$-forms.
}
\smallskip

In studying cohomological decomposition of the de Rham cohomology of almost-K\"ahler manifolds, the third author introduced a {\em Lefschetz type property} for $2$-forms, see Definition \ref{defi:lefschetz-type-property}. Such a property is stronger than the Hard Lefschetz Condition on $2$-classes, namely, the property that $\left[\omega\right]^{n-2}\smile\sspace \colon H^{2}_{dR}(M;\R)\to H^{2n-2}_{dR}(M;\R)$ is an isomorphism, where $2n:=\dim M$.

We study such a Lefschetz type property on almost-K\"ahler manifolds $\left(M,\, J,\, \omega,\, g\right)$ in relation to the existence of a cohomological decomposition of $H^2_{dR}(M;\R)$. More precisely, we prove the following result.

\smallskip
\noindent {\bfseries Theorem \ref{thm:cf-implies-lefschetz}.\ }{\itshape
 Let $\left(M,\,J,\,\omega,\,g\right)$ be a compact almost-K\"ahler manifold. Suppose that there exists a basis of $H^2_{dR}(X;\R)$ represented by $g$-harmonic $2$-forms which are of pure type with respect to $J$. Then the Lefschetz type property on $2$-forms is satisfied.
}
\smallskip

Note that, by the hypothesis, it follows,in particular, that $J$ is \Cpf\ and \pf, \cite[Theorem 3.7]{fino-tomassini}.
Note also that A. Fino and the second author provided in \cite{fino-tomassini} several examples of compact non-K\"ahler solvmanifolds admitting a basis of harmonic representatives of pure-type with respect to the almost-complex structure.
In \cite[\S2]{draghici-li-zhang-2}, T. Dr\v{a}ghici, T.-J. Li, and the third author ask whether such a Lefschetz type property on $2$-forms is actually equivalent to \Cf ness for every almost-K\"ahler nilmanifold and solvmanifold, without any further assumption; Theorem \ref{thm:cf-implies-lefschetz} and Proposition \ref{prop:iwasawa-non-full} provide results and examples in favour of a possibly positive answer to their question.

\smallskip

In \cite[Theorem 1.1]{draghici-li-zhang-1}, starting with a compact complex surface $\left(M,\,J\right)$, it is shown that the dimension $h^-_{\tilde J}$ of the ${\tilde J}$-anti-invariant cohomology subgroup $H^-_{\tilde{J}}(M)$ of any {\em metric related} almost-complex structure $\tilde{J}$ on $M$ (namely, an almost-complex structure ${\tilde J}$ on $M$ inducing the same orientation as that one induced by $J$ and with a common compatible metric), such that $\tilde{J}\neq\pm J$, can be $0$, $1$, or $2$, and a description of such almost-complex structures $\tilde{J}$ having $h^-_{\tilde J}\in \{1,\, 2\}$ is provided.
Furthermore, it is conjectured that $h^-_J=0$ for a generic almost-complex structure $J$ on a compact $4$-dimensional manifold, and that if $h^-_J\geq 3$, then $J$ is integrable, \cite[Conjecture 2.4, Conjecture 2.5]{draghici-li-zhang-1}. One could set a similar question for higher dimensional manifolds, asking Question \ref{question:large}: {\em are there examples of non-integrable almost-complex structures $J$ on a compact $2n$-dimensional manifold with $h^-_J > n\, (n-1)$?}.

Finally, we prove a Nomizu-type result for the subgroups $H^\pm_J(M)$ of a completely-solvable solvmanifolds $M=\left.\Gamma\middle\backslash G\right.$ endowed with left-invariant almost-complex structures $J$. More precisely, denote the Lie algebra associated to $G$ by $\mathfrak{g}$, and consider
$$ H^{(p,q),(q,p)}_J(\mathfrak{g})_\R \;:=\; \left\{\mathfrak{a}=\left[\alpha\right]\in H^\bullet\left(\wedge^\bullet\mathfrak{g}^*,\, \de\right) \st \alpha \in \wedge^{(p,q),(q,p)}_J\mathfrak{g}^* \right\} \;\subseteq\; H^\bullet_{dR}(M;\R) $$
the subgroup of $H^\bullet_{dR}(M;\R)$ that consists of classes admitting a left-invariant representative of type $(p,q)+(q,p)$, where $\wedge^{(p,q),(q,p)}_J\mathfrak{g}^* := \left(\wedge^{p,q}\left(\mathfrak{g}\otimes_\R\C\right)^*\oplus \wedge^{q,p}\left(\mathfrak{g}\otimes_\R\C\right)^*\right) \cap \wedge^\bullet\mathfrak{g}^*$; then the following result holds.

\smallskip
\noindent {\bfseries Theorem \ref{thm:nomizu-type}.\ }{\itshape
  Let $M=\left.\Gamma\middle\backslash G\right.$ be a solvmanifold endowed with a left-invariant almost-complex structure $J$, and denote the Lie algebra naturally associated to $G$ by $\mathfrak{g}$. For any $p,q\in\N$, the map $j\colon H^{(p,q),(q,p)}_J(\mathfrak{g})_\R \to H^{(p,q),(q,p)}_J(M)_\R$ induced by left-translations is injective, and, if $H_{dR}^\bullet\left(\wedge^\bullet\mathfrak{g}^*,\, \de\right) \simeq H^\bullet_{dR}(M;\R)$ (for instance, if $M$ is a completely-solvable solvmanifold), then $j\colon H^{(p,q),(q,p)}_J(\mathfrak{g})_\R \to H^{(p,q),(q,p)}_J(M)_\R$ is in fact an isomorphism.
}
\smallskip

In particular, it follows that $\dim_\R H^-_J(M) \leq n\,(n-1)$ for every left-invariant almost-complex structure on a completely-solvable solvmanifold.

\smallskip

\noindent{\itshape Acknowledgements.} The authors would like to thank Tedi Dr\v{a}ghici for useful comments and remarks and Tian-Jun Li for helpful discussions. The first author would like to thank also Jean-Pierre Demailly for useful conversations and for his warm hospitality at Insitut Fourier, Universit\'{e} de Grenoble \textsc{i}. The second author would like to thank the Department of Mathematics of University of Notre Dame and the School of Mathematics of University of Minnesota for their warm hospitality. The third author would like to thank Institut des Hautes \'Etudes Scientifiques for providing excellent research environment. We are also pleased to thank the referee for fruitful suggestions and remarks.

\section{$\mathcal{C}^\infty$-pure-and-full almost-complex structures}
\subsection{Subgroups of the de Rham cohomology of an almost-complex manifolds}
We start by fixing some notation and recalling some recent results on cohomological properties of almost-complex manifolds; for more details
see, e.g., \cite{li-zhang, draghici-li-zhang, draghici-li-zhang-1, fino-tomassini, angella-tomassini, angella-tomassini-1, draghici-li-zhang-2}, and the references therein.

Let $J$ be a smooth almost-complex structure on a compact $2n$-dimensional manifold $M$. Denote by $\wedge^rM$ the bundle of $r$-forms on $M$; we denote with the same symbol $\wedge^rM:=\Gamma(M,\wedge^rM)$ the space of smooth global sections of the bundle $\wedge^rM$. Then $J$ extends
to a complex automorphism of $T^\C M=TM\otimes\C$ such that $T^\C M =T^{1,0}_JM\oplus T^{0,1}_JM$, where $T^{1,0}_JM$ and $T^{0,1}_JM$ are the $\left(\pm \im\right)$-eigenbundles. The action of $J$ can be extended to the space $\wedge^r(M;\C)$
of smooth global sections of the bundle $\wedge^r (M;\C):=\wedge^rM\otimes\C$ getting the following decomposition:
$$
\wedge^r(M;\C)=\displaystyle{\bigoplus_{p+q=r}}\wedge^{p,q}_JM\;.
$$
Then the space $\wedge^rM$ of real smooth differential $r$-forms decomposes as
$$
\wedge^rM =\displaystyle{\bigoplus_{p + q = r,\;p\leq q}} \wedge^{(p,q),(q,p)}_J (M)_{\R} \;,
$$
where, for $p<q$, (later on, we do not distinguish the cases $p<q$ and $p=q$,)
$$ \wedge^{(p,q),(q,p)}_J (M)_{\R} := \left\{ \alpha \in \wedge^{p,q}_J M \oplus \wedge^{q,p}_J M
\st \alpha = \overline \alpha \right\} \;, \quad \wedge^{(p,p)}_J (M)_{\R} := \left\{ \alpha \in \wedge^{p,p}_J M
\st \alpha = \overline \alpha \right\} $$
In particular, for $r=2$, we will adopt the following notation:
$$
\wedge^{1,1}_J(M)_\R=:\wedge^+_JM\;,\quad\wedge^{(2,0),(0,2)}_J(M)_\R=:\wedge^-_JM \;;
$$
this is consistent with the decomposition in invariant and anti-invariant part of $\wedge^2M$ under the natural action of $J$
on $\wedge^2M$, given by $J\alpha(\sspace,\,\sspace):=\alpha(J\sspace,\,J\sspace)$.\\
We will refer to forms in $\wedge^{1,1}_J(M)_\R$, respectively $\wedge^{(2,0),(0,2)}_J(M)_\R$ as forms of {\em pure type with respect to $J$}.

\smallskip

For a finite set $S$ of pairs of integers, let
$$
{\mathcal Z}^{S}_J := \displaystyle{\bigoplus_{(p, q) \in S,\;p\leq q}} {\mathcal Z}^{(p,q),(q,p)}_J\;, \qquad
{\mathcal B}^{S}_J := \bigoplus_{(p, q) \in S,\;p\leq q} {\mathcal B}^{(p,q),(q,p)}_J \;,
$$
where
\begin{eqnarray*}
{\mathcal Z}^{(p,q),(q,p)}_J &:=& \left\{\alpha\in\wedge^{(p,q),(q,p)}_J (M)_{\R} \st \de\alpha=0\right\} \;,\\[5pt]
{\mathcal B}^{(p,q),(q,p)}_J &:=& \left\{\beta\in\wedge^{(p,q),(q,p)}_J (M)_{\R} \st
\text{ there exists }\gamma \text{ such that }
\de\gamma =\beta\right\} \;.
\end{eqnarray*}
Define
$$
H^S_J (M)_{\R} := \frac{{\mathcal Z}^{S}_J}{{\mathcal B}^S_J} \;.
$$
Let $\mathcal B$ be the space of $\de$-exact forms. Since $\frac{{\mathcal Z}^{S}_J}{{\mathcal B}^S_J}=\frac{{\mathcal Z}^{S}_J}{\mathcal{B}\cap{\mathcal Z}^{S}_J}$, a natural inclusion $\rho_S\colon \frac{{\mathcal Z}^{S} _J}{{\mathcal B}^{S}_J} \to \frac{{\mathcal Z}^{S}_J}{{\mathcal B}}$ is defined.
As in \cite{li-zhang}, we will write
$\rho_S \left(\frac{{\mathcal Z}^{S}_J}{{\mathcal B}^{S}_J}\right)$ simply as $\frac{{\mathcal Z}^{S}_J}{{\mathcal B}^{S}_J}$
and consequently the cohomology spaces $H^S_J (M)_{\R}$ can be identified as
$$
H^S_J (M)_{\R}= \left\{ [\alpha] \in H^\bullet_{dR}(M;\R) \st \alpha \in {\mathcal Z}^{S} _J \right\} =
\frac {{\mathcal Z}^{S} _J } {\mathcal B} \;.
$$
Therefore, there is a natural inclusion
$$
H^{(1,1)}_J (M)_{\R} + H^{(2,0),(0,2)}_J (M)_{\R} \subseteq
H^{2}_{dR} (M; \R) \;.
$$

\subsection{\Cpf\ and \pf\ almost-complex structures}
As in \cite{li-zhang}, we set the following definition.

\begin{defi}[{\cite[Definition 2.2, Definition 2.3, Lemma 2.2]{li-zhang}}]
An almost-complex structure $J$ on a manifold $M$ is said to be
\begin{itemize}
\item {\em \Cp} if $H^{(1,1)}_J (M)_{\R} \cap H^{(2,0),(0,2)}_J (M)_{\R} = \left\{0\right\}$,
\item {\em \Cf} if $H^{2}_{dR} (M; \R) = H^{(1,1)}_J (M)_{\R} + H^{(2,0),(0,2)}_J (M)_{\R}$,
\item {\em \Cpf} if
$$ H^{2}_{dR} (M; \R) \;=\; H^{(1,1)}_J (M)_{\R} \oplus H^{(2,0),(0,2)}_J (M)_{\R} \;. $$
\end{itemize}
\end{defi}

According to the previous notation, we will write
$$ H^+_J(M) \;:=\; H^{(1,1)}_J (M)_{\R} \;, \qquad H^-_J(M) \;:=\; H^{(2,0),(0,2)}_J (M)_{\R} \;.$$

Similar definitions in terms of currents can be given, introducing the notion of \emph{\pf} almost-complex structure: we refer to \cite[\S2.2.2]{li-zhang} for further details and results. More precisely, on an almost complex manifold $\left(M, \, J\right)$, the space ${\mathcal E}_k (M)_{\R}$ of real $k$-currents has a decomposition ${\mathcal E}_k (M)_{\R} = \bigoplus_{\substack{p + q = k \\ p\leq q}} {\mathcal E}_{(p,q),(q,p)}^J (M)_{\R}$, where ${\mathcal E}_{(p,q),(q,p)}^J (M)_{\R}$ denotes the space of real $k$-currents of bi-dimension $(p,q)+(q,p)$.

Let ${\mathcal Z}_{(2,0),(0,2)}^J$ and ${\mathcal Z}_{(1,1)}^J$ denote the spaces of real $\de$-closed currents of bi-dimension $(2,0)+(0,2)$, respectively $(1,1)$, and ${\mathcal B}_{(2,0),(0,2)}^J$ and ${\mathcal B}_{(1,1)}^J$ denote the spaces of real $\de$-exact currents of bi-dimension $(2,0)+(0,2)$, respectively $(1,1)$. Denote by ${\mathcal B}$ the space of boundaries. Let, as in \cite{li-zhang},
\begin{eqnarray*}
H_{(1,1)}^J (M)_{\R} &:=& \left\{ [\alpha ]\in H_2(M;\R) \st \alpha \in {\mathcal Z}_{(1,1)}^J \right\} \;=\; \frac{{\mathcal Z}_{(1,1)} ^J }
{\mathcal{B}} \;, \\[5pt]
H_{(2,0),(0,2)}^J (M)_{\R} &:=& \left\{ [\alpha ]\in H_2(M;\R) \st \alpha \in {\mathcal Z}_{(2,0),(0,2)}^J \right\} \;=\; \frac {{\mathcal Z}_{(2,0),(0,2)} ^J}{\mathcal{B}} \;.
\end{eqnarray*}

We recall the following definition.

\begin{defi}[{\cite[Definition 2.15, Definition 2.16]{li-zhang}}]
An almost complex structure $J$ on a manifold $M$ is said to be {\em \p} if $H_{(1,1)}^J (M)_{\R} \cap H_{(2,0),(0,2)}^J (M)_{\R} = \left\{0\right\}$. It is said to be {\em \f} if $H_2(M;\R)=H_{(1,1)}^J (M)_{\R} + H_{(2,0),(0,2)}^J (M)_{\R}$. Therefore, an almost complex structure $J$ is {\em \pf} if and only if
$$ H_{2} (M, \R) \;=\; H_{(1,1)}^J (M)_{\R} \oplus H_{(2,0),(0,2)}^J (M)_{\R} \;. $$
\end{defi}

\smallskip

In \cite[Proposition 2.1]{li-zhang} it is shown that, given a compact complex manifold $\left(M,\, J\right)$ of complex dimension $n$,
if $n=2$ or $J$ is K\"ahler, then $J$ is \Cpf, and $H^{(1,1)}_J(M)_\R  \simeq H^{1,1}_{\delbar}(M) \cap H^2_{dR}\left(M;\R\right)$ and $H^{(2,0),(0,2)}_J(M)_\R \simeq \left(H^{2,0}_{\delbar}(M)\oplus H^{0,2}_{\delbar}(M)\right) \cap H^2_{dR}\left(M;\R\right)$.
In view of this result, the subgroups $H^{(1,1)}_J(M)_\R$ and $H^{(2,0),(0,2)}_J(M)_\R$ of the de Rham cohomology can be viewed as an analogue of the Dolbeault cohomology groups for non-integrable almost-complex structures.

In \cite[Theorem 2.3]{draghici-li-zhang} it is proven the following result.
\begin{thm}[{\cite[Theorem 2.3]{draghici-li-zhang}}]\label{thm:dlz}
If $M$ is a compact manifold
of dimension $4$, then any almost-complex structure $J$ on $M$ is \Cpf.
\end{thm}
This is no more true in dimension higher than $4$: in \cite[Example 3.3]{fino-tomassini}, a compact non-\Cp\ almost-complex structure on a $6$-dimensional nilmanifold is constructed. Therefore, the previous theorem can be considered a sort of Hodge decomposition theorem in the non-K\"ahler case.

\section{Cohomological properties of almost-K\"ahler manifolds}\label{sec:cohomological-symplectic-properties}
\subsection{Lefschetz type property on almost-K\"ahler manifolds with pure-type harmonic representatives}
Given a compact $2n$-dimensional almost-K\"ahler manifold $\left(M,\, J,\, \omega,\, g\right)$, we are interested in studying the property of being \Cpf.

\smallskip

Firstly we recall the following result.
\begin{prop}[{\cite[Proposition 2.8]{draghici-li-zhang}, \cite[Proposition 3.2]{fino-tomassini}}]\label{prop:almK-Cp}
If $J$ is an almost-complex structure on a compact manifold $M$ and $J$ admits a compatible symplectic structure, then $J$ is \Cp.
\end{prop}
Furthermore, A. Fino and the second author proved that an almost-K\"ahler manifold admitting a basis of harmonic $2$-forms whose elements are of pure type with respect to the almost-complex structure is \Cpf\ and \pf, \cite[Theorem 3.7]{fino-tomassini}; they also provided several examples of compact non-K\"ahler solvmanifolds satisfying the above assumption in \cite{fino-tomassini}.

\smallskip

To the purpose of studying the property of being \Cpf\ on almost-K\"ahler manifolds, we recall the following definition.
\begin{defi}\label{defi:lefschetz-type-property}
Given a compact $2n$-dimensional symplectic manifold $\left(M,\,\omega\right)$, denote by
$$
\mathcal{L}_\omega\colon\wedge^2M\to\wedge^{2n-2}M\,,\qquad \mathcal{L}_\omega(\alpha):=\omega^{n-2}\wedge \alpha\,,
$$
the {\em Lefschetz type operator} (on $2$-forms) associated with $\omega$. \newline 
Then one says that the compact $2n$-dimensional almost-K\"ahler manifold $\left(M,\,J,\, \omega,\, g\right)$ satisfies the {\em Lefschetz type property (on $2$-forms)} if $\mathcal{L}_\omega$ 
takes $g$-harmonic $2$-forms to $g$-harmonic $(2n-2)$-forms.
\end{defi}

Furthermore, we recall some notions and results from \cite{brylinski, mathieu, yan}, see also \cite{merkulov, cavalcanti}.\\
Let $\left(M,\,\omega\right)$ be a compact $2n$-dimensional symplectic manifold. Extend $\omega^{-1}\colon T^*M\to TM$ to the whole exterior algebra of $T^*M$. For any $k\in\N$, the \emph{symplectic $\star_\omega$ operator} is defined as
$$ \star_\omega\colon\wedge^{k}M\to \wedge^{2n-k}M \;, \qquad \beta\wedge\star_\omega\alpha\;=\;\omega^{-1}\left(\alpha,\beta\right)\,\frac{\omega^n}{n!} \;, \quad \forall \alpha,\,\beta\in\wedge^kM \;.$$
One can prove that $\star_\omega^2=\id_{\wedge^\bullet M}$, \cite[Lemma 2.1.2]{brylinski}.\\
For any $k\in\N$, define the \emph{symplectic co-differential operator}
$$ \delta_\omega\colon \wedge^kM\to\wedge^{k-1}M\;,\qquad \delta_\omega\lfloor_{\wedge^kM}\;:=\; \left(-1\right)^{k+1}\,\star_\omega\,\de\,\star_\omega \;;$$
this operator has been studied by J.-L. Brylinski in \cite{brylinski} for Poisson manifolds; in the context of generalized complex geometry, see, e.g., \cite{gualtieri}, it can be interpreted as the symplectic counterpart of the operator $\de^c:=-\im\left(\del-\delbar\right)$ in complex geometry, see \cite{cavalcanti}.\\
By definition, $\left(M,\,\omega\right)$ satisfies the \emph{Hard Lefschetz Condition} if, for each $k\in\N$, the map
$$ \left[\omega\right]^k\smile \sspace \colon H^{n-k}_{dR}(M;\R)\to H^{n+k}_{dR}(M;\R) $$
is an isomorphism. O. Mathieu, \cite[Corollary 2]{mathieu}, and, independently, D. Yan, \cite[Theorem 0.1]{yan}, proved that, given a compact symplectic manifold $\left(M,\,\omega\right)$, any de Rham cohomology class has a (possibly non-unique) \emph{$\omega$-symplectically harmonic representative} (that is, a $\de$-closed $\delta_\omega$-closed representative) if and only if the Hard Lefschetz Condition holds.

\smallskip

We can now prove the following result.
\begin{thm}\label{thm:cf-implies-lefschetz}
 Let $\left(M,\,J,\,\omega,\,g\right)$ be a compact almost-K\"ahler manifold. Suppose that there exists a basis of $H^2_{dR}(X;\R)$ represented by 
 $g$-harmonic $2$-forms which are of pure type with respect to $J$. Then the Lefschetz type property on $2$-forms is satisfied.
\end{thm}

\begin{proof}
Recall that, on a $2n$-dimensional almost-K\"ahler manifold $\left(M,\,J,\,\omega,\,g\right)$, the Hodge $*_g$ operator and the symplectic $\star_\omega$ operator are related by $\star_\omega \;=\; *_g\,J$, \cite[Theorem 2.4.1, Remark 2.4.4]{brylinski}. Therefore, for forms of pure type with respect to $J$, the properties of being $g$-harmonic and of being $\omega$-symplectically harmonic are equivalent. The theorem follows noting that, \cite[Lemma 1.2]{yan}, $\left[\mathcal{L}_\omega,\,\de\right] = 0$ and $\left[\mathcal{L}_\omega,\,\delta_\omega\right] = \de$, hence $\mathcal{L}_\omega$ sends $\omega$-symplectically harmonic $2$-forms (of pure type with respect to $J$) to $\omega$-symplectically harmonic $(2n-2)$-forms (of pure type with respect to $J$).
\end{proof}

\smallskip

\begin{rem} We note that if $\left(M,\,J,\,\omega, \,g \right)$ is a compact $2n$-dimensional almost-K\"ahler manifold satisfying the 
 Lefschetz type property on $2$-forms and $J$ is \Cf, then $J$ is \Cpf\ and \pf.

Indeed, we have already remarked that $J$ is \Cp, see Proposition \ref{prop:almK-Cp}. Moreover, since $J$ is 
 \Cf, $J$ is also \p\ by \cite[Proposition 2.5]{li-zhang}. We recall now the argument in \cite{fino-tomassini} to prove that $J$ is also \f.

Firstly, note that if the Lefschetz type property on $2$-forms holds, then 
$\left[\omega^{n-2}\right]\smile\sspace\colon H^{2}_{dR}\left(M;\R\right)\to H^{2n-2}_{dR}\left(M;\R\right)$ 
is an isomorphism.
Therefore, we get that
$$ H^{2n-2}_{dR}(M;\R) \;=\; H^{(n,n-2),(n-2,n)}_J(M)_\R + H^{(n-1,n-1)}_J(M)_\R \;; $$
indeed, (following the argument in \cite[Theorem 4.1]{fino-tomassini},) since 
$\left[\omega^{n-2}\right]\smile\sspace\colon H^2_{dR}(M;\R)\to H^{2n-2}_{dR}(M;\R)$ is in particular surjective, we have
\begin{eqnarray*}
H^{2n-2}_{dR}(M;\R) &=& \left[\omega^{n-2}\right]\smile H^{2}_{dR}(M;\R) \;=\; \left[\omega^{n-2}\right] \smile \left(H^{(2,0),(0,2)}_J(M)_\R\oplus H^{(1,1)}_J(M)_\R\right) \\[2pt]
&\subseteq& H^{(n,n-2),(n-2,n)}_J(M)_\R + H^{(n-1,n-1)}_J(M)_\R \;,
\end{eqnarray*}
yielding the above decomposition of $H^{2n-2}_{dR}(M;\R)$.
Then, it follows that $J$ is also \f, see, for example, \cite[Theorem 2.1]{angella-tomassini}.
\end{rem}

\subsection{A family of almost-K\"ahler manifolds satisfying the Lefschetz type property on $2$-forms}
Let $\mathfrak{n}$ be the $6$-dimensional nilpotent Lie algebra whose structure equations, with respect to a basis
$\left\{e^j\right\}_{j\in\{1,\ldots,6\}}$ of $\mathfrak{n}^*$, are given by
$$ \de e^1 \;=\; \de e^2 \;=\; \de e^3 \;=\; 0 \;, \qquad \de e^4 \;=\; e^{23} \;, \qquad \de e^5 \;=\; e^{13} \;, \qquad \de e^6 \;=\; e^{12}
$$
(where we write $e^{jk}$ instead of $e^j\wedge e^k$). Using a result by Mal'tsev, \cite[Theorem 7]{malcev}, the connected simply-connected
Lie group $G$ associated with $\mathfrak{n}$ admits a discrete co-compact subgroup $\Gamma$: let $N:=\left.\Gamma\right\backslash G$
be the (compact) nilmanifold obtained as a quotient of $G$ by $\Gamma$. Note that $N$ is not formal by a theorem of K. Hasegawa's, \cite[Theorem 1, Corollary]{hasegawa-89}.

\smallskip

\noindent Fix $\alpha>1$ and take
$$ \omega_\alpha \;:=\; e^{14}+\alpha\cdot e^{25}+\left(\alpha-1\right)\cdot e^{36} \;;$$
since $\de\omega_\alpha=0$ and $\omega_\alpha^3\neq 0$, we get that $\omega_\alpha$ is a left-invariant symplectic form on $N$. Set
$$
\begin{array}{lll}
J_\alpha \, e_1 \;:=\; e_4 \;, & J_\alpha\, e_2\;:=\;\alpha\, e_5\;, &
J_\alpha \, e_3\;:=\;(\alpha -1)\,e_6\;,\,\,\,\,\vspace{.2cm}\\
J_\alpha \,e_4\;:=\;-e_1\;,&
J_\alpha \,e_5\;:=\;-\frac{1}{\alpha}\,e_2\;, &
J_\alpha\, e_6\;:=\;-\frac{1}{\alpha -1}\,e_3\;,
\end{array}
$$
where $\left\{e_1,\, \ldots,\, e_6\right\}$ denotes the global dual frame of $\left\{e^1, ,\ldots,\, e^6\right\}$ on N. It is immediate to check that
\begin{itemize}
\item setting $g_{\alpha}(\cdot,\cdot)\,:=\,\omega_\alpha(\cdot, J_\alpha\cdot)$, the triple $\left(J_\alpha,\,\omega_\alpha,\,g_{\alpha}\right)$ gives rise to a family of left-invariant almost-K\"ahler structures on $N$;
\item denoting by
\begin{eqnarray*}
 E_\alpha^1 \;:=\; e^1 \;, \qquad & E_\alpha^2 \;:=\; \alpha\, e^2 \;, & \qquad E_\alpha^3 \;:=\; (\alpha-1)\, e^3 \;, \\[5pt]
 E_\alpha^4 \;:=\; e^4 \;, \qquad & E_\alpha^5 \;:=\; e^5 \;,          & \qquad E_\alpha^6 \;:=\; e^6\,,
\end{eqnarray*}
then $\left\{E_\alpha^1,\ldots,E_\alpha^6\right\}$ is a $g_{\alpha}$-orthonormal co-frame on $N$; with respect to this new co-frame, we easily obtain the following structure equations:
$$  \de E_\alpha^1 = \de E_\alpha^2 = \de E_\alpha^3 = 0 , \quad \de E_\alpha^4 = \frac{1}{\alpha(\alpha -1)} \, E_\alpha^{23} , \quad \de E_\alpha^5 = \frac{1}{\alpha -1}\, E_\alpha^{13} , \quad \de E_\alpha^6 = \frac{1}{\alpha}\,E_\alpha^{12} . $$
\end{itemize}
Then,
$$
\varphi_\alpha^1 := E_\alpha^1+\im E_\alpha^4 \,,\qquad
\varphi_\alpha^2 := E_\alpha^2+\im E_\alpha^5 \,,\qquad
\varphi_\alpha^3 := E_\alpha^3+\im E_\alpha^6\;,
$$
are $(1,0)$-forms with respect to the almost-complex structure $J_\alpha$, and
$$
\omega_\alpha \;=\; E_\alpha^{14} + E_\alpha^{25} +E_\alpha^{36} \;.
$$

\smallskip

\noindent By a result of K. Nomizu's, \cite[Theorem 1]{nomizu}, see Theorem \ref{thm:nomizu-hattori}, the de Rham cohomology of $N$ is straightforwardly computed:
$$  H^2_{dR}(N;\R) \;\simeq\; \R\left\langle E_\alpha^{15},\; E_\alpha^{16},\; E_\alpha^{24},\; E_\alpha^{26},\; E_\alpha^{34},\; E_\alpha^{35},\; E_\alpha^{14}+\frac{1}{\alpha}\,E_\alpha^{25},\; \frac{1}{\alpha}\,E_\alpha^{25}+\frac{1}{\alpha-1}\,E_\alpha^{36} \right\rangle $$
(where we have listed the $g_{\alpha}$-harmonic representatives instead of their classes). \newline
Note that the listed $g_{\alpha}$-harmonic representatives of $H^2_{dR}(N;\R)$ are of pure type with respect to $J_\alpha$: hence, the almost-complex structure $J_\alpha$ is \Cpf\ by
\cite[Theorem 3.7]{fino-tomassini}; in particular, note that
\begin{eqnarray*}
 H^2_{dR}(N;\R) & \simeq & \R\left\langle \im\,\alpha\,\varphi_\alpha^{1\bar1}+\im\,\varphi_\alpha^{2\bar2},\; \im\,(\alpha-1)\,\varphi_\alpha^{2\bar2}+
 \im\,\alpha\,\varphi_\alpha^{3\bar3},\;\Im\mathfrak{m}\,\varphi_\alpha^{1\bar2},\;\Im\mathfrak{m}\,\varphi_\alpha^{1\bar3},\;\Im\mathfrak{m}\,\varphi_\alpha^{3\bar2} \right\rangle \\[2pt]
&&\oplus \left\langle \Im\mathfrak{m}\,\varphi_\alpha^{12},\;\Im\mathfrak{m}\,\varphi_\alpha^{13},\;\Im\mathfrak{m}\,\varphi_\alpha^{23} \right\rangle \;,
\end{eqnarray*}
hence $h^+_{J_\alpha}(N) = 5$ and $h^-_{J_\alpha}(N) = 3$.

Moreover, one explicitly notes that
$$
\begin{array}{rccclrcccl}
 \mathcal{L}_{\omega_\alpha}E_\alpha^{15} &=& E_\alpha^{1536} &=& *_{g_{\alpha}}\, E_\alpha^{24}\,, \qquad &\mathcal{L}_{\omega_\alpha}E_\alpha^{16} &=& E_\alpha^{1625} &=& *_{g_{\alpha}}\, E_\alpha^{34} \;,\\[2pt]
 \mathcal{L}_{\omega_\alpha}E_\alpha^{24} &=& E_\alpha^{2436} &=& *_{g_{\alpha}}\, E_\alpha^{15}\,, \qquad &
 \mathcal{L}_{\omega_\alpha}E_\alpha^{26} &=& E_\alpha^{2614} &=& *_{g_{\alpha}}\, E_\alpha^{35} \;,\\[2pt]
 \mathcal{L}_{\omega_\alpha}E_\alpha^{34} &=& E_\alpha^{3425} &=& *_{g_{\alpha}}\, E_\alpha^{16}\,, \qquad &
 \mathcal{L}_{\omega_\alpha}E_\alpha^{35} &=& E_\alpha^{3514} &=& *_{g_{\alpha}}\, E_\alpha^{26} \;,
\end{array}
$$
while
$$ \mathcal{L}_{\omega_\alpha}\,\left(E_\alpha^{14}+\frac{1}{\alpha}\,E_\alpha^{25}\right)\;=\; -\frac{\alpha+1}{\alpha}\, E_\alpha^{1245}-\frac{1}{\alpha}\, E_\alpha^{2356}-E_\alpha^{1346} $$
where
\begin{eqnarray*}
\de\, *_{g_{\alpha}}\,\mathcal{L}_{\omega_\alpha}\,\left(E_\alpha^{14}+\frac{1}{\alpha}\,E_\alpha^{25}\right) &=& \de\left(-\frac{\alpha+1}{\alpha}\, E_\alpha^{36}-
E_\alpha^{25}-\frac{1}{\alpha}\,E_\alpha^{14}\right)=0 \;,
\end{eqnarray*}
and, by a similar computation, $\de\, *_{g_{\alpha}}\,\mathcal{L}_{\omega_\alpha}\,\left(e^{25}+e^{36}\right) = 0$.
This proves explicitly that $\omega_\alpha$ satisfies the Lefschetz type property on $2$-forms.

The nilmanifold $N$ is not formal by a theorem of K. Hasegawa's, \cite[Theorem 1, Corollary]{hasegawa-89}. The non-formality of $M$ can be also proved by giving a non-zero triple Massey product on $N$, see \cite{d-g-m-s}: since
$$ \left[E_\alpha^1\right]\smile \left[E_\alpha^3\right] \;=\; \left(\alpha-1\right)\,\left[\de E_\alpha^5\right]\;=\; 0 \;, \quad \left[E_\alpha^3\right]\smile \left[E_\alpha^2\right] \;=\; -\alpha\,\left(\alpha-1\right)\,\left[\de E_\alpha^4\right]\;=\; 0 \;,$$
we get that the triple Massey product
$$
\left\langle [E_\alpha^1],\, [E_\alpha^3],\, [E_\alpha^2] \right\rangle \;=\; -\left(\alpha-1\right) \, \left[ E_\alpha^{25}+\alpha\,E_\alpha^{14}\right]
$$
does not vanish, and hence $N$ is not formal.

In summary, we have proven the following result.
\begin{prop}
There is a non-formal $6$-dimensional nilmanifold $N$ endowed with a $1$-parameter family $\left\{\left(J_\alpha,\, \omega_\alpha,\, g_{\alpha}\right)\right\}_{\alpha>1}$ of left-invariant almost-K\"ahler structures being \Cpf\ and \pf\ and satisfying the Lefschetz type property on $2$-forms.
\end{prop}

\begin{rem}
It has to be noted that $\omega_\alpha\wedge\sspace\colon \wedge^2N^6\to\wedge^4N^6$ induces an isomorphism in cohomology $\left[\omega_\alpha\right]\smile\sspace\colon H^2_{dR}(N,\R)\to H^4_{dR}(N,\R)$, while, accordingly to \cite[Theorem A]{benson-gordon1}, $\left[\omega_\alpha\right]^2\smile\sspace \colon H^1_{dR}(N,\R)\to H^5_{dR}(N,\R)$ is not an isomorphism.
\end{rem}

\section{Almost-K\"ahler \Cpf\ structures}
\subsection{The Nakamura manifold of completely solvable type}
Take $A\in\SL(2;\Z)$ with two different real eigenvalues $\esp^\lambda$ and $\esp^{-\lambda}$ with $\lambda>0$, and fix $P\in\GL(2;\R)$ such that $PAP^{-1}=\mathrm{diag}\left(\esp^{\lambda},\,\esp^{-\lambda}\right)$. For example, take
$$ A:=
\left(
\begin{array}{cc}
 2 & 1 \\
 1 & 1
\end{array}
\right)
\quad \text{ and } \quad
P:=
\left(
\begin{array}{cc}
 \frac{1-\sqrt{5}}{2} & 1 \\
 1 & \frac{\sqrt{5}-1}{2}
\end{array}
\right)
$$
and consequently $\lambda=\log\frac{3+\sqrt{5}}{2}$.
Let $M^6:=:M^6(\lambda)$ be the compact complex manifold
$$ M^6 := \mathbb{S}^1_{x^2} \times \frac{\R_{x^1}\times\T^2_{\C,\,\left(x^3,\,x^4,\,x^5,\,x^6\right)}}
{\left\langle T_1\right\rangle} $$
where $\T^2_\C$ is the $2$-dimensional complex torus $\T^2_\C := \frac{\C^2}{P\Z\left[\im\right]^2}$ and $T_1$ acts on $\R\times \T^2_\C$ as
$T_1\left(x^1,\,x^3,\,x^4,\,x^5,\,x^6\right) :=\left(x^1+\lambda,\, \esp^{-\lambda}x^3,\, \esp^{\lambda}x^4,\,
\esp^{-\lambda}x^5,\, \esp^{\lambda}x^6\right)$.
The manifold $M^6$ can be seen as a compact quotient of a completely-solvable Lie group by a discrete co-compact subgroup, \cite[Example 3.1]{fernandez-munoz-santisteban}; (denote the Lie algebra naturally associated to the completely-solvable Lie group of $M^6$ by $\mathfrak{g}$).
Using coordinates $x^2$ on $\mathbb{S}^1$, $x^1$ on $\R$ and $\left(x^3,x^4,x^5,x^6\right)$ on $\T^2_\C$, we set
$$ e^1 := \de x^1 , \;\; e^2 := \de x^2 , \;\; e^3 := \esp^{x^1}\de x^3 , \;\; e^4 := \esp^{-x^1}\de x^4 , \;\; e^5 := \esp^{x^1}\de x^5 , \;\; e^6 := \esp^{-x^1}\de x^6 $$
as a basis for $\mathfrak{g}^*$; therefore, with respect to $\{e^i\}_{i\in\{1,\ldots,6\}}$, the structure equations are the following:
$$ \de e^1 \;=\; \de e^2 \;=\; 0 \;, \quad \de e^3 \;=\; e^{13} \;, \quad \de e^4 \;=\; -e^{14} \;, \quad \de e^5 \;=\; e^{15} \;, \quad \de e^6 \;=\; -e^{16} \;. $$

\subsection{The de Rham cohomology of the Nakamura manifold}\label{subsec:derham-m6}

Let $J$ be the almost-complex structure on $M^6$ defined by the complex $(1,0)$-forms given by
$$ \varphi^1 \;:=\; \frac{1}{2}\left(e^1+\im e^2\right)\;, \qquad \varphi^2 \;:=\; e^3+\im e^5\;, \qquad \varphi^3 \;:=\; e^4+\im e^6 \;. $$
It is straightforward to check that $J$ is integrable.\\
Being $M^6$ a compact quotient of a completely-solvable Lie group, one computes the de Rham cohomology of $M^6$ easily by
A. Hattori's theorem \cite[Corollary 4.2]{hattori}, see Theorem \ref{thm:nomizu-hattori}:
\begin{eqnarray*}
H^1_{dR}\left(M^6;\C\right) &\simeq& \C\left\langle \varphi^1, \, \bar\varphi^1 \right\rangle \;, \quad
H^2_{dR}\left(M^6;\C\right) \;\simeq\; \C\left\langle \varphi^{1\bar1}, \, \varphi^{2\bar3}, \, \varphi^{3\bar2}, \,
\varphi^{23}, \, \varphi^{\bar2\bar3}  \right\rangle \;, \\[2pt]
H^3_{dR}\left(M^6;\C\right) &\simeq& \C\left\langle \varphi^{12\bar3}, \, \varphi^{13\bar2}, \, \varphi^{123}, \,
\varphi^{1\bar2\bar3}, \, \varphi^{2\bar1\bar3}, \, \varphi^{3\bar1\bar2}, \, \varphi^{23\bar1}, \, \varphi^{\bar1\bar2\bar3}
\right\rangle
\end{eqnarray*}
(for the sake of clearness, we write, for example, $\varphi^{A\bar B}$ in place of $\varphi^A\wedge\bar\varphi^B$ and we list the harmonic representatives with respect to the metric $g:=\sum_{j=1}^{3} \varphi^j\odot\bar\varphi^j$ instead of their classes). Therefore, $M^6$ is {\em geometrically formal}, i.e., the product of $g$-harmonic forms is still $g$-harmonic, and therefore it is {\em formal}, namely the de Rham complex of $M$ is formal as a differential graded algebra, see, e.g., \cite{d-g-m-s}. Furthermore, it can be easily checked that
$$
\omega \;:=\; e^{12}+e^{34}+e^{56}
$$
gives rise to a symplectic structure on $M^6$ satisfying the Hard Lefschetz Condition. We obtain the following result.
\begin{prop}[{\cite[Proposition 3.2]{fernandez-munoz-santisteban}}]
 The manifold $M^6$ is formal and it admits a symplectic form $\omega$ satisfying the Hard Lefschetz Condition.
\end{prop}

\noindent Note also that $\tilde\omega:=\frac{\im}{2}\left(\varphi^{1\bar1}+\varphi^{2\bar2}+\varphi^{3\bar3}\right)$
is not $\de$-closed but $\de\tilde\omega^{2}=0$, from which it follows that the manifold $M^6$ admits a balanced metric.

Moreover, since $M^6$ is a compact quotient of a completely-solvable Lie group, by the K. Hasegawa's theorem \cite[Main Theorem]{hasegawa-06},
we have the following result, see also \cite[Theorem 3.3]{fernandez-munoz-santisteban}. (We recall that a compact complex manifold is said to belong to \emph{class $\mathcal{C}$ of Fujiki} if it admits a proper modification from a K\"ahler manifold.)

\begin{thm}[{\cite[Main Theorem]{hasegawa-06}}]
 The manifold $M^6$ admits no K\"ahler structure and it is not in class $\mathcal{C}$ of Fujiki.
\end{thm}

\subsection{An almost-K\"ahler structure on the Nakamura manifold}
By K. Hasegawa's theorem \cite[Main Theorem]{hasegawa-06}, any integrable complex structure on $M^6$ (for example, the $J$ defined in
\S\ref{subsec:derham-m6}) does not admit any symplectic structure compatible with it. Therefore, we consider the almost-complex structure $J'$ defined by
$$ J'e^1\;:=\;-e^2\;,\qquad J' e^3\;:=\;-e^4\;,\qquad J' e^5\;:=\;-e^6\;;$$
considering
$$
\psi^1 \;:=\; \frac{1}{2}\left(e^1+\im e^2\right) \;,\quad
\psi^2 \;:=\; e^3+\im e^4 \;,\quad
\psi^3 \;:=\; e^5+\im e^6
$$
as a co-frame for the space of $(1,0)$-forms on $\left(M^6,\,J'\right)$, one can compute
$$
 \de\psi^1 = 0\,,\quad
 \de\psi^2 = \psi^{1\bar2}+\psi^{\bar1\bar2}\,,\quad
 \de\psi^3 = \psi^{1\bar3}+\psi^{\bar1\bar3}\;,
$$
from which it is clear that $J'$ is not integrable. Note that the $J'$-compatible $2$-form
$$ \omega':=e^{12}+e^{34}+e^{56} $$
is $\de$-closed. Hence, $\left(M^6,\,J',\,\omega'\right)$ is an almost-K\"ahler manifold.\\
Moreover, recall that
$$ H^2_{dR}\left(M^6;\R\right) \;\simeq\; \underbrace{\R\left\langle \im\psi^{1\bar1},\, \im\psi^{2\bar2},\, \im\psi^{3\bar3},\, \im\left(\psi^{2\bar3}+\psi^{3\bar2}\right) \right\rangle}_{\subseteq H^+_{J'}\left(M^6\right)_\R} \oplus \underbrace{\R\left\langle \im\left(\psi^{23}-\psi^{\bar2\bar3}\right)\right\rangle}_{\subseteq H^-_{J'}\left(M^6\right)_\R} \;, $$
where we have listed the harmonic representatives with respect to the metric $g':=\sum_{j=1}^6 e^{j}\odot e^{j}$
instead of their classes; note that the listed $g'$-harmonic representatives are of pure type with respect to $J'$. Therefore, $J'$ is obviously \Cf; it is also \Cp\ by \cite[Proposition 3.2]{fino-tomassini} or \cite[Proposition 2.8]{draghici-li-zhang}, see Proposition \ref{prop:almK-Cp}. Moreover, since any cohomology class in $H^+_{J'}\left(M^6\right)_\R$
(respectively, in $H^-_{J'}\left(M^6\right)_\R$) has a $g'$-harmonic representative in $\mathcal{Z}^{(1,1)}_{J'}$
(respectively, in $\mathcal{Z}^{(2,0),(0,2)}_{J'}$), by \cite[Theorem 3.7]{fino-tomassini} we have that $J'$ is also \pf. One can explicitly check that the Lefschetz type operator $\mathcal{L}_{\omega'}\colon\wedge^2M^6\to\wedge^4M^6$ introduced in \S\ref{sec:cohomological-symplectic-properties} takes $g'$-harmonic $2$-forms to $g'$-harmonic $4$-forms, since
$$
\begin{array}{lr}
 \mathcal{L}_{\omega'}\,e^{12} = e^{1234}+e^{1256} = *_{g'}\,\left(e^{34}+e^{56}\right) \;, &\quad \mathcal{L}_{\omega'}\,e^{36} = e^{1236} = *_{g'}\,e^{45} \;,\\[2pt]
 \mathcal{L}_{\omega'}\,e^{34} = e^{1234}+e^{3456} = *_{g'}\,\left(e^{12}+e^{56}\right) \;, &\quad \mathcal{L}_{\omega'}\,e^{45} = e^{1245} = *_{g'}\,e^{36} \;, \\[2pt]
 \mathcal{L}_{\omega'}\,e^{56} = e^{1256}+e^{3456} = *_{g'}\,\left(e^{12}+e^{34}\right) \;. & \\[2pt]
\end{array}
$$

Resuming, we have shown the following result.
\begin{prop}
Let $M^6$ be the Nakamura manifold. Then there exist a complex structure $J$ and an almost-K\"ahler structure $\left(J',\,\omega',\,g'\right)$, both of which are \Cpf\ and \pf.\\
Furthermore, the Lefschetz type operator of the almost-K\"ahler structure $\left(J',\,\omega',\,g'\right)$ takes $g'$-harmonic $2$-forms to $g'$-harmonic $4$-forms.
\end{prop}

Inspired by the argument of the proof of \cite[Theorem 2.3]{draghici-li-zhang}, see Theorem \ref{thm:dlz}, one can ask the following question, compare also \cite[\S2]{draghici-li-zhang-2}; we provide in Proposition \ref{prop:iwasawa-non-full} an example of a non-\Cf\ almost-K\"ahler structure for which the Lefschetz type property on $2$-forms does not hold.
\begin{question}
Let $\left(M,\, J,\, \omega,\, g\right)$ be a compact $2n$-dimensional almost-K\"ahler manifold satisfying the Lefschetz type property on $2$-forms. Is $J$ \Cf?
\end{question}

\section{An almost-K\"ahler non-\Cf\ structure}
Let $X:=\left. \Z\left[\im\right]^3 \right\backslash \left(\C^3,\,*\right)$ be the {\em Iwasawa manifold}, where the group structure on $\C^3$ is defined by
$$ \left(z_1,\,z_2,\,z_3\right) * \left(w_1,\,w_2,\,w_3\right) := \left(z_1+w_1 ,\, z_2+w_2 ,\, z_3+z_1w_2+w_3\right) \;.$$
Considering the standard complex structure induced by the one on $\C^3$ and setting $\left\{\varphi^1,\,\varphi^2,\,\varphi^3\right\}$ as a global co-frame for the $(1,0)$-forms on $X$, by A. Hattori's theorem \cite[Corollary 4.2]{hattori}, see Theorem \ref{thm:nomizu-hattori}, one gets that
\begin{eqnarray*}
 H^2_{dR}(X;\C) &\simeq& \R\left\langle \varphi^{13}+\varphi^{\bar{1}\bar{3}},\; \im\left(\varphi^{13}-\varphi^{\bar{1}\bar{3}}\right),\;\varphi^{23}+\varphi^{\bar{2}\bar{3}},\; \im\left(\varphi^{23}-\varphi^{\bar{2}\bar{3}}\right)\right.,\\[2pt]
  && \left.\varphi^{1\bar{2}}-\varphi^{2\bar{1}},\; \im\left(\varphi^{1\bar{2}}+\varphi^{2\bar{1}}\right),\;\im\varphi^{1\bar{1}},\;\im\varphi^{2\bar{2}} \right\rangle \,\otimes_\R\, \C \;,
  \end{eqnarray*}
where we have listed the harmonic representatives with
respect to the metric $g:=\sum_{h=1}^3\varphi^h\odot\bar\varphi^h$ instead of their classes.
Set
$$
\varphi^{1}=:e^1+\im e^2\;,\qquad \varphi^{2}=:e^3+\im e^4\;,\qquad \varphi^{3}=:e^5+\im e^6\;;
$$
then,
$$
\de e^5 \;=\; -e^{13}+e^{24} \;, \qquad \de e^6 \;=\; -e^{14}-e^{23} \;,
$$
the other differentials being zero. Therefore,
$$ H^2_{dR}(X;\R) \;\simeq\; \R\left\langle e^{15}-e^{26},\, e^{16}+e^{25},\, e^{35}-e^{46},\, e^{36}+e^{45},\, e^{13}+e^{24},\, e^{23}-e^{14},\, e^{12},\, e^{34} \right\rangle \;.$$
Set
\begin{eqnarray*}
&{}& v_1:=e^{15}-e^{26}\;,\qquad v_2:=e^{16}+e^{25}\;,\qquad v_3:= e^{35}-e^{46}\;,\qquad v_4:=e^{36}+e^{45} \;,\\[2pt]
 &{}& v_5:=e^{13}+e^{24}\;,\qquad v_6:=e^{23}-e^{14}\;,\qquad v_7:=e^{12}\;,\qquad v_8:=e^{34}\;.\end{eqnarray*}

\smallskip

Consider the almost-K\"ahler structure $\left(J,\, \omega,\, g\right)$ on $X$ defined by
$$ Je^1 \;:=\; -e^6\;,\qquad  Je^2 \;:=\; -e^5\;,\qquad Je^3 \;:=\; -e^4 \;, \qquad \omega \;:=\; e^{16}+e^{25}+e^{34}\;. $$

We easily get that
$$
 \R\left\langle v_2,\,v_3+v_5,\,v_4-v_6,\,v_8\right\rangle \subseteq H^+_J(X) \;, \qquad  \R\left\langle v_1,\,v_3-v_5,\,v_4+v_6\right\rangle \subseteq H^-_J(X)\;.
$$
We claim that the previous inclusions are actually equalities, and in particular that \emph{$J$ is a non-\Cf\ almost-K\"ahler structure} on $X$.\\
Indeed, we firstly note that, by \cite[Proposition 3.2]{fino-tomassini} or \cite[Proposition 2.8]{draghici-li-zhang}, see Proposition \ref{prop:almK-Cp}, $J$ is \Cp, since it admits a symplectic structure compatible with it. Moreover, we recall that a \Cf\ almost-complex structure is also \p\ by \cite[Proposition 2.30]{li-zhang} and therefore it satisfies also that
\begin{equation}\label{eq:pure-4}
H^{(3,1),(1,3)}_J(X)_\R \cap H^{(2,2)}_J(X)_\R \;=\; \left\{0\right\} \;, 
\end{equation}
see \cite[Theorem 2.4]{angella-tomassini}. Therefore, our claim reduces to prove that $J$ does not satisfy \eqref{eq:pure-4}. Note that
\begin{eqnarray*}
 \left[e^{3456}\right] &=& \left[ e^{3456}-\de e^{135}\right]=\left[e^{3456}+e^{1234}\right] \\[2pt]
                       &=& \left[ e^{3456}+\de e^{135}\right]=\left[e^{3456}-e^{1234}\right]
\end{eqnarray*}
and that $e^{3456}+e^{1234}\in\wedge^{(3,1),(1,3)}_J(X)_\R$ while $e^{3456}-e^{1234}\in\wedge^{(2,2)}_J(X)_\R$, and so
$H^{(3,1),(1,3)}_J(X)_\R\cap H^{(2,2)}_J(X)_\R\ni \left[e^{3456}\right]$, therefore \eqref{eq:pure-4} does not hold, and hence $J$ is not \Cf.

Let $\mathcal{L}_\omega$ be the Lefschetz type operator of the almost-K\"ahler structure $\left(J,\, \omega,\, g\right)$. Then, we have
$\mathcal{L}_\omega\left(e^{12}\right)=e^{1234}=\de\left(e^{245}\right)$, i.e., $\mathcal{L}_\omega$ does not take $g$-harmonic $2$-forms in $g$-harmonic $4$-forms.

Hence, we have proved the following result.
\begin{prop}\label{prop:iwasawa-non-full}
Let $X:=\left. \Z\left[\im\right]^3 \right\backslash \left(\C^3,\,*\right)$ be the real manifold underlying the Iwasawa manifold. Then there exists an almost-K\"ahler structure $\left(J,\,\omega,\,g\right)$ on $X$ which is \Cp\ and non-\Cf.\\
Furthermore, the Lefschetz type operator of the almost-K\"ahler structure $\left(J,\,\omega,\, g\right)$ does not take $g$-harmonic $2$-forms to $g$-harmonic $4$-forms.
\end{prop}

\section{Almost-complex manifolds with large anti-invariant cohomology}
Given an almost-complex structure $J$ on a compact manifold $M$, it is natural to ask how large the cohomology subgroup $H^{(2,0),(0,2)}_J(M)_\R$ can be. In this direction, T. Dr\v{a}ghici, T.-J. Li, and the third author raised the following question in \cite{draghici-li-zhang-1}.
\begin{question}[{\cite[Conjecture 2.5]{draghici-li-zhang-1}}]
Are there compact $4$-dimensional manifold $M$ endowed with non-integrable almost-complex structures $J$ such that $\dim_\R H^-_J(M) \geq 3$?
\end{question}

\smallskip

We present here a $1$-parameter family $\left\{J_t\right\}_t$ of (non-integrable) almost-complex structures on the $6$-dimensional torus $\T^6$ having $h^-_{J_t}:=\dim_\R H^{-}_{J_t}\left(\T^6\right)_\R$ greater than $3$, see also \cite[\S4]{angella-tomassini}. For $t$ small enough, set $\alpha_t:=:\alpha_t\left(x^3\right)\in\mathcal{C}^\infty\left(\T^6\right)$ such that $\alpha_0(x^3)\equiv 1$ and set
$$
\varphi^1_t \;:=\; \de x^1\,+\,\im\,\alpha_t\,\de x^4\,,\quad
\varphi^2_t \;:=\; \de x^2\,+\,\im\,\de x^5\,,\quad
\varphi^3_t \;:=\; \de x^3\,+\,\im\,\de x^6\;;
$$
therefore, the structure equations are
$$
\de\varphi^1_t \;=\; \im\,\de\alpha_t\,\wedge\,\de x^4\,,\quad
\de\varphi^2_t \;=\; 0\,,\quad
\de\varphi^3_t \;=\; 0 \;.
$$
Straightforward computations give that the $J$-anti-invariant $\de$-closed $2$-forms are of the type
$$ \psi \;=\; \frac{C}{\alpha_t}\,\left(\de x^{13}-\alpha_t\,\de x^{46}\right) \,+\, D\,\left(\de x^{16} - \alpha_t\, \de x^{34}\right) +\, E\, \left(\de x^{23} - \de x^{56}\right) \,+\, F\,\left(\de x^{26} - \de x^{35}\right) \;, $$
where $C,\;D,\;E,\;F\in\R$ (we shorten $\de x^j\wedge \de x^k$ by $\de x^{jk}$). Moreover, the forms $\de x^{23} - \de x^{56}$ and $\de x^{26} - \de x^{35}$ are clearly harmonic with respect to the standard flat metric $\sum_{j=1}^{6}\de x^j\otimes\de x^j$, while the classes of $\de x^{16} - \alpha_t\, \de x^{34}$ and $\de x^{13}-\alpha_t\,\de x^{46}$ are non-zero, their harmonic parts being non-zero. Hence, we get that $h^-_{J_0}=6$ and
$$ h^-_{J_t} \;=\; 4 \qquad \text{ for small }\qquad t\neq0 \;. $$

\smallskip

In the general case, we ask the following natural question.
\begin{question}\label{question:large}
 Are there examples of non-integrable almost-complex structures $J$ on a compact $2n$-dimensional manifold with $\dim_\R H^-_J(M) > n\, (n-1)$?
\end{question}

\smallskip

Consider now a solvmanifold $M=\left. \Gamma \right\backslash G$, namely, a compact quotient of a connected simply-connected solvable Lie group $G$ by a co-compact discrete subgroup $\Gamma$. Denote the Lie algebra naturally associated to $G$ by $\mathfrak{g}$, and consider $\left(\wedge^\bullet\mathfrak{g}^*,\, \de\right)$ the subcomplex of the de Rham complex $\left(\wedge^\bullet M,\, \de\right)$ given by the left-invariant differential forms. The following result by K. Nomizu \cite{nomizu} and A. Hattori \cite{hattori} holds.

\begin{thm}[{\cite[Theorem 1]{nomizu}, \cite[Theorem 4.2]{hattori}}]\label{thm:nomizu-hattori}
Let $M$ be a nilmanifold or, more in general, a completely-solvable solvmanifold. Then $H^\bullet\left(\wedge^\bullet\mathfrak{g}^*,\, \de\right) \simeq H^\bullet_{dR}(M;\R)$.
\end{thm}

Let $J$ be a left-invariant almost-complex structure on $M$, namely, an almost-complex structure on $M$ induced by an almost-complex structure on $G$ that is invariant under the action of $G$ on itself given by left-translations.
Given $p,q\in\N$, denote by
$$ H^{(p,q),(q,p)}_J(\mathfrak{g})_\R \;:=\; \left\{\mathfrak{a}=\left[\alpha\right]\in H^\bullet\left(\wedge^\bullet\mathfrak{g}^*,\, \de\right) \st \alpha \in \wedge^{(p,q),(q,p)}_J\mathfrak{g}^* \right\} \;\subseteq\; H^\bullet_{dR}(M;\R) $$
the subgroup (see, e.g., \cite[Lemma 9]{console-fino}) of $H^\bullet_{dR}(M;\R)$ that consists of classes admitting a left-invariant representative of type $(p,q)+(q,p)$, where $\wedge^{(p,q),(q,p)}_J\mathfrak{g}^* := \left(\wedge^{p,q}\left(\mathfrak{g}\otimes_\R\C\right)^*\oplus \wedge^{q,p}\left(\mathfrak{g}\otimes_\R\C\right)^*\right) \cap \wedge^\bullet\mathfrak{g}^*$.

Using Belgun's symmetrization trick, \cite[Theorem 7]{belgun}, one can prove the following Nomizu-type result, which relates the subgroups $H^{(p,q),(q,p)}_J\left(M\right)_\R$ with their left-invariant part $H^{(p,q),(q,p)}_J(\mathfrak{g})_\R$.

\begin{thm}\label{thm:nomizu-type}
Let $M=\left.\Gamma\right\backslash G$ be a solvmanifold endowed with a left-invariant almost-complex structure $J$, and denote the Lie algebra naturally associated to $G$ by $\mathfrak{g}$. For any $p,q\in\N$, the map
$$ j\colon H^{(p,q),(q,p)}_J(\mathfrak{g})_\R \to H^{(p,q),(q,p)}_J(M)_\R $$
induced by left-translations is injective, and, if $H_{dR}^\bullet\left(\wedge^\bullet\mathfrak{g}^*,\, \de\right) \simeq H^\bullet_{dR}(M;\R)$ (for instance, if $M$ is a completely-solvable solvmanifold), then $j\colon H^{(p,q),(q,p)}_J(\mathfrak{g})_\R \to H^{(p,q),(q,p)}_J(M)_\R$ is in fact an isomorphism.
\end{thm}

\begin{proof}
Since $J$ is left-invariant, left-translations induce the map $j\colon H^{(p,q),(q,p)}_J(\mathfrak{g})_\R \to H^{(p,q),(q,p)}_J(M)_\R$.

Since, by J. Milnor's Lemma \cite[Lemma 6.2]{milnor}, $G$ is unimodular, one can take in particular a bi-invariant volume form $\eta$ on $M$ such that $\int_M\eta=1$. Consider the F.~A. Belgun symmetrization map in \cite[Theorem 7]{belgun}, namely,
$$ \mu\colon \wedge^\bullet M \to \wedge^\bullet{\mathfrak{g}}^*\;,\qquad \mu(\alpha)\;:=\;\int_M \alpha\lfloor_m \, \eta(m) \;.$$

Since $\mu$ commutes with $\de$ by \cite[Theorem 7]{belgun}, it induces the map $\mu\colon H^\bullet_{dR}(M;\R) \to H^\bullet\left(\wedge^\bullet\mathfrak{g}^*,\, \de\right)$, and, since $\mu$ commutes with $J$, it preserves the bi-graduation; therefore it induces the map $\mu\colon H^{(p,q),(q,p)}_J(M) \to H^{(p,q),(q,p)}_J(\mathfrak{g})_\R$. Moreover, since $\mu$ is the identity on the space of left-invariant forms by \cite[Theorem 7]{belgun}, we get the commutative diagram
$$
\xymatrix{
H^{(p,q),(q,p)}_{J}(\mathfrak{g})_\R \ar[r]^{j} \ar@/_1.5pc/[rr]_{\id} & H^{(p,q),(q,p)}_{J}(M)_\R \ar[r]^{\mu} & H^{(p,q),(q,p)}_{J}(\mathfrak{g})_\R
}
$$
hence $j\colon H^{(p,q),(q,p)}_{J}(\mathfrak{g})_\R \to H^{(p,q),(q,p)}_{J}(M)_\R$ is injective, and $\mu\colon H^{(p,q),(q,p)}_{J}(M)_\R \to H^{(p,q),(q,p)}_{J}(\mathfrak{g})_\R$ is surjective.

Furthermore, when $H^\bullet\left(\wedge^\bullet\mathfrak{g}^*,\, \de\right) \simeq H^\bullet_{dR}(M;\R)$ (for instance, when $M$ is a completely-solvable solvmanifold, by A. Hattori's theorem \cite[Theorem 4.2]{hattori}, see Theorem \ref{thm:nomizu-hattori}), since $\mu\lfloor_{\wedge^\bullet{\mathfrak{g}}^*} = \id\lfloor_{\wedge^\bullet{\mathfrak{g}}^*}$ by \cite[Theorem 7]{belgun}, we get that $\mu\colon H^\bullet_{dR}(M;\R) \to H^\bullet\left(\wedge^\bullet\mathfrak{g}^*,\, \de\right)$ is the identity map, and hence $\mu\colon H^{(p,q),(q,p)}_{J}(M)_\R \to H^{(p,q),(q,p)}_{J}(\mathfrak{g})_\R$ is also injective, and hence an isomorphism.
\end{proof}

In particular, if $M=\left.\Gamma\right\backslash G$ is a $2n$-dimensional completely-solvable solvmanifold endowed with a left-invariant almost-complex structure $J$, then
$$ \dim_\R H^-_J(M) \;\leq\; n\, (n-1) \qquad \text{ and } \qquad \dim_\R H^+_J(M) \;\leq\; n^2 \;;$$
this provides a partial negative answer to Question \ref{question:large}.

\end{document}